\documentclass[a4paper,12pt]{article}
\usepackage[utf8]{inputenc}

\usepackage[a4paper,margin=2.3cm]{geometry}
\usepackage{amsmath}
\usepackage{amssymb}
\usepackage{amsthm}

\usepackage{tikz}

\newtheorem{thm}{Theorem}

\newtheorem{pro}{Proposition}
\newtheorem{lem}{Lemma}
\newtheorem*{principle}{Ader's Non-separation Principle}

\def\cone{\mathrm{cone}\,}
\def\R{\mathbb{R}}
\def\T{\hat T}
\newcommand{\Orth}{\mathbf{O}(2)}
\newcommand{\PD}[2]{\mathbf{PD}(#1,#2)}

\title{On the Banach-Mazur Ellipse}
\author{V. D. Babev\thanks{Faculty of Mathematics and Informatics, Sofia University, 5, James Bourchier Blvd, 1164 Sofia, Bulgaria; babev@fmi.uni-sofia.bg. Research partially supported by Grant No. 80-10-116 from 28.05.2025 of the Research Fund of Sofia University.}, M. Ivanov\thanks{Radiant Life Technologies Ltd., Nicosia, Cyprus; milen@radiant-life-technologies.com. Research partially supported by the Bulgarian National Science Fund under Grant No. KP-06-H92/6 (December 8, 2025).}, \fbox{R. Nikolov}} 
\date{February, 2026}

\begin{document}

\maketitle

\begin{abstract}
  We provide a new proof of Ader's \cite{ader} characterisation of the ellipse of minimal Banach-Mazur distance to the unit circle of a normed plane in terms of contact and extremal points.

  Our method reveals the relation of this problem to the Chebyshev alternance.
\end{abstract}

\medskip\noindent\textbf{Keywords and phrases:} Banach-Mazur distance, normed plane, distance ellipse, Chebyshev alternance

\medskip\noindent\textbf{2020 Mathematics Subject Classification:} Primary 46B20; Secondary 52A10, 41A50

\section{Introduction}
Recall that the Banach-Mazur distance between the isomorphic normed spaces $X$ and $Y$ is denoted by 
\begin{equation}
  \label{eq:def-bmd}
  d(X,Y) := \inf \{\|T\|\cdot\|T^{-1}\|:\ T:X\to Y\ \text{isomorphism}\},
\end{equation}
for details see e.g.\ \cite[Section 6.1.1]{pietsch}.

If $X$ and $Y$ are of one and the same finite dimension, then the minimum in \eqref{eq:def-bmd} is attained.

An important characteristic of a finite dimensional space is its distance to the Euclidean space. Here we study this characteristic in the simplest case: the plane.

So, for a 2-dimensional normed space $(X,\|\cdot\|)$
$$
  d_2(X) := d(X,\R^2),
$$
where $\R^2$ is considered with the standard Euclidean norm $|\cdot|$. It is easy to check that
\begin{equation}
  \label{eq:d2-inclusion}
  d_2(X) = \min \{\|T\|:\ T:X\to \R^2,\ T(B_X)\supset B_{\R^2}\},
\end{equation}
where $B_X$ and $S_X$ stand for the unit ball and the unit sphere of $X$ respectively. Indeed, $T(B_X)\supset B_{\R^2}$ is equivalent to $\|T^{-1}\| \le 1$, so we have that $d_2(X)$ is at most the right-hand side of \eqref{eq:d2-inclusion}. On the other hand, multiplying the operator $T$ by a real constant does not change $\|T\|\cdot\|T^{-1}\|$, so we may assume that $\|T^{-1}\| = 1$ in \eqref{eq:def-bmd}.

Using \textit{polar decomposition}, we can reduce \eqref{eq:d2-inclusion}. Note that if $U$ is in $\Orth$ -- the group of isometries of $\R^2$, that is, $U^{-1} = U^t$ if considered as a matrix -- then $\|UT\| = \max_{x\in S_X} |UTx| = \max_{x\in S_X} |Tx| = \|T\|$, and similarly $\|T^{-1}U^{-1}\| = \|T^{-1}\|$, so we can quotient by $\Orth$. Denote by $\PD{X}{\R^2}$ the operators from $X$ to $\R^2$ that are positive definite symmetric matrices in any basis. By \cite[Theorem~7.3.1, p.449]{hj} each isomorphism $T:X\to\R^2$ can be represented as $T=UT_1$, where $U\in\Orth$ and $T_1\in\PD{X}{\R^2}$. Therefore,
\begin{equation}
  \label{eq:d2-reduction}
  d_2(X) = \min \{\|T\|:\ T\in\PD{X}{\R^2},\ T(B_X)\supset B_{\R^2}\}.
\end{equation}

We prove in a new way the following characterisation, which is essentially contained in \cite{ader,mory}.
\begin{thm}
  \label{thm:main}
  Let $\dim X =2$. There is a unique $\T\in\PD{X}{\R^2}$ such that $\T(B_X)\supset B_{\R^2}$ and
  \begin{equation}
    \label{eq:tm=d2}
    \|\T\| = d_2(X).
  \end{equation}
  $\T$ is characterised by the following property:

  There are $x_i\in S_X$, $i=1,2$, with $x_1\neq \pm x_2$, and $y_i\in S_X$, $i=1,2$, with $y_1\neq \pm y_2$, such that:
  \begin{equation}
    \label{eq:x-distance}
    |\T x_i| = \|\T\|,\quad i=1,2,
  \end{equation}
  \begin{equation}
    \label{eq:y-contact}
    |\T y_i| = 1,\quad i=1,2,
  \end{equation}
  \begin{equation}
    \label{eq:x-between-y}
    y_1\in\cone \{x_1,x_2\},\quad y_2\in\cone\{-x_1,x_2\}.
  \end{equation}
\end{thm}
Here
$$
  \cone \{x,y\} := \{\alpha x + \beta y:\ \alpha,\beta \ge 0\}.
$$
Note that the condition characterising the optimal operator depends only on the operator itself and not on $d_2(X)$. This is to be expected, because the problem suitably stated is convex. Not strictly convex, however, so the uniqueness is not immediate.

As an illustration, consider the following picture.

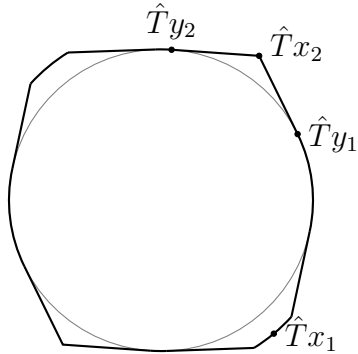
\begin{center}
  \begin{tikzpicture}
    \draw[color=gray] (0,0) circle [radius = 2cm];
    \coordinate (y1) at (86:2);
    \coordinate (y1m) at (180+86:2);
    \filldraw (y1) circle (1pt);
    \node[above] at (y1) {$\hat Ty_2$};

    \coordinate (y2) at (26:2);
    \coordinate (y2m) at (180+26:2);
    \filldraw (y2) circle (1pt);
    \node[right] at (y2) {$\hat Ty_1$};

    \coordinate (x1) at (56:4/1.73205080757);
    \coordinate (x1m) at (180+56:4/1.73205080757);
    \filldraw (x1) circle (1pt);
    \node[right, yshift=1.6mm] at (x1) {$\hat Tx_2$};

    \draw[thick] (y1) -- (x1) -- (y2);
    \draw[thick] (y1m) -- (x1m) -- (y2m);

    \coordinate (a) at (-12:2);
    \coordinate (am) at (180-12:2);
    \coordinate (b) at (-42:4/1.73205080757);
    \coordinate (bm) at (180-42:4/1.73205080757);

    \draw[thick] (a) -- (b);
    \draw[thick] (am) -- (bm);
    \draw[thick] (y2) arc(26:-12:2);
    \draw[thick] (y2m) arc(180+26:180-12:2);

    \coordinate (u) at (272:2);
    \coordinate (um) at (92:2);
    \draw[thick] (y1m) arc(266:272:2);
    \draw[thick] (y1) arc(86:92:2);

    \coordinate (v) at (302:4/1.73205080757);
    \draw[thick] (u) -- (v);
    \coordinate (vm) at (122:4/1.73205080757);
    \draw[thick] (um) -- (vm);

    \draw[thick] (v) arc(302:318:4/1.73205080757);
    \draw[thick] (vm) arc(122:138:4/1.73205080757);

    \coordinate (x2) at (310:4/1.73205080757);
    \filldraw (x2) circle (1pt);
    \node[right] at (x2) {$\hat T x_1$};
  \end{tikzpicture}
\end{center}

From it we can also see one important, but immediate, corollary. Since one of the angles between $\T y_1$ and $\T y_2$, and between $\T y_2$ and $-\T y_1$ (say the one between $\T y_1$ and $\T y_2$ as on the picture) will be $\le \pi/2$; and since $\T x_2$ will be within the angle between the tangents to the unit circle at $\T y_1$ and $\T y_2$, due to convexity, the distance from $\T x_2$ to zero, that is $d_2(X)$, is at most $\sqrt{2}$. Therefore,
$$
  d_2(X) \le \sqrt{2}.
$$

Note that Theorem~\ref{thm:main} easily follows from
\begin{principle}[\cite{ader}]
  For the unique $\T$ there is no angle such that all points $x\in S_X$ with $|\T x| = \|\T\|$ are inside this angle and all points $y\in S_X$ with $|\T y| = 1$ are outside.
\end{principle}
Here, an angle is
$$
  \cone \{a,b\} \cup \cone \{-a,-b\}
$$
for some $a,b\in\R^2$.

The original work \cite{ader} states that there should be at most $4$ points realising Ader's Non-separation Principle, and for these $4$ points it is then clear that the alternance described in Theorem~\ref{thm:main} must be in place. By Maurey's result \cite{mory}, see also \cite{gru-ko}, the ellipse of minimal Banach-Mazur distance is unique, which yields the uniqueness part of Theorem~\ref{thm:main}.

Obviously, the problem can be re-stated more geometrically like: find the minimal $k\ge1$ such that there is an ellipse inscribed in $B_X$ whose homothetic image with ratio $k$ is circumscribed about $B_X$. We will actually be using this form below, but we mention it here to point out the relationship to the celebrated L\"owner--John ellipse, see e.g.\ \cite{john}, that is, the ellipse of maximal area inscribed in $B_X$. The L\"owner--John ellipsoid is extensively used for estimating the Banach-Mazur distance to Euclidean space, probably because Ader's result \cite{ader} had apparently been forgotten until Grundbacher and Kobos recently unearthed it, see \cite{gru-ko}.

However, if $B_X$ is not symmetric the estimate through the L\"owner--John ellipsoid can grow progressively worse with the dimension. Consider $\R^n$ and let $B_{X_n}$ be the convex hull of $B_{\R^n}$ and $\sqrt{n}e$, where $e\in S_{\R^n}$. Then $B_{\R^n}$ is the L\"owner--John ellipsoid for $B_{X_n}$, see \cite{rakov} for a nice and simple proof, so it gives only the generic estimate $\sqrt{n}$. On the other hand, by the rotational symmetry of $B_{X_n}$ around the axis through $e$, $d(X_n,\R^n)\le\sqrt{3}$.

Our interest in the specifically two-dimensional case comes from our previous works \cite{iv-tr,iv-pa-tr,nik1,nik2}. Of course, similar problems can also be considered in different metrics, like e.g.\ the Hausdorff, see \cite{kenderov,gru-ke,ke-ki}.

In conclusion, note that the convexity of $B_X$ will not be used at all in the proofs below, which is not surprising, as they go through Chebyshev approximation.

\section{An equivalent problem in polar co-ordinates}
  \label{sec:polar}
  In view of \eqref{eq:d2-reduction}, Theorem~\ref{thm:main} characterises the solution of the optimisation problem 
  $$
    \begin{cases}
      \|T\| \to \min\\
      T \in \PD{X}{\R^2},\quad \|T^{-1}\| \le 1.
    \end{cases}
  $$
  By substituting $T^{-1}$ in place of $T$, this problem can be rewritten as
  \begin{equation}
    \label{eq:t-1-prob}
    \begin{cases}
      \|T^{-1}\| \to \min\\
      T \in \PD{\R^2}{X},\quad \|T\| \le 1,
    \end{cases}
  \end{equation}
  which is more suited for our purposes. Theorem~\ref{thm:main} is immediately equivalent to the following Proposition~\ref{prop:equiv} with which we will be working from now on.
  \begin{pro}
    \label{prop:equiv}
    The only solution to \eqref{eq:t-1-prob}, say $\tilde T$, is characterised by the following property:

    There are $u_i,v_i\in S_{\R^2}$, $i=1,2$, with $u_1\neq \pm u_2$, $v_1\neq \pm v_2$, $v_1\in\cone \{u_1,u_2\}$, $v_2\in\cone \{-u_1,u_2\}$, and
    \begin{equation}
      \label{eq:equiv-alternance}
      \|\tilde Tu_i\| = 1/\|\tilde T^{-1}\|,\quad \|\tilde T v_i\| = 1,\quad i=1,2.
    \end{equation}
  \end{pro}
  Let a co-ordinate system be fixed in $X$ in which
  \begin{equation}
    \label{eq:r-*}
    S_X = \{r(\varphi)(\cos\varphi,\sin\varphi):\ \varphi\in\R\},
  \end{equation}
  for a strictly positive, $\pi$-periodic and continuous $r:\R\to\R^+$.

  In these fixed coordinates each $T\in\PD{\R^2}{X}$ is identified with a matrix
  $$
    T = \begin{pmatrix} 
      a & b \\ b & c 
    \end{pmatrix},\quad a,c>0,\quad \det T = ac-b^2>0.
  $$
  Obviously,
  $$
    x \in T(S_{\mathbb{R}^2}) \iff |T^{-1}x| = 1 \iff x^t T^{-2} x = 1.
  $$
  That is,  $T(S_{\mathbb{R}^2})$ is the ellipse:
  $$
    a_1 x_1^2 + 2 b_1 x_1x_2 + c_1 x_2^2 = 1,
  $$
  where
  $$
    \begin{pmatrix} a_1 & b_1 \\ b_1 & c_1 \end{pmatrix} = T^{-2},
  $$
  so $a_1,c_1 > 0$ and $b_1^2 < a_1c_1$. In polar co-ordinates then $T(S_{\mathbb{R}^2}) = \{\rho(\varphi)(\cos\varphi,\sin\varphi):\ \varphi\in\R\}$, where
  \begin{equation}
    \label{eq:rho-trig}
    \rho^{-2}(\varphi) = a_2 + b_2\cos 2\varphi + c_2\sin 2\varphi,
  \end{equation}
  with $a_2 = (a_1+c_1)/2$, $b_2 = (a_1-c_1)/2$ and $c_2 = b_1$, so
  \begin{equation}
    \label{eq:ellipse-cond}
    a_2 > 0,\quad a_2^2 > b_2^2 + c_2^2.
  \end{equation}
  Usually this is expressed in terms of $\cos2(\varphi-\theta)$, where $\theta$ is the tilt angle. Let $\mathcal E$ be the family of all ellipses in polar co-ordinates, that is, all curves in the form \eqref{eq:rho-trig} for all values of the parameters $a_2,b_2,c_2$ which satisfy \eqref{eq:ellipse-cond}. It is clear that the cone \eqref{eq:ellipse-cond} parametrises one-to-one the ellipses, as well as $\PD{\R^2}{X}$, which we will use when establishing the uniqueness part of our main result. Let
  \begin{equation}
    \label{eq:C-def}
    \mathcal{C} := \{\log \rho:\ \rho\in\mathcal{E}\}.
  \end{equation}
  Let also, see \eqref{eq:r-*},
  $$
    f := \log r.
  $$
  Since $\|T\| \le 1$ is equivalent to $T(S_{\mathbb{R}^2})\subset B_X$, which in polar co-ordinates translates to
  $$
    \rho \le r \iff g := \log \rho \le f,
  $$
  the constraint in \eqref{eq:t-1-prob} becomes $g\in\mathcal{C}$, $g \le f$. For the objective, since $\rho(\varphi)(\cos\varphi,\sin\varphi) \in T(S_{\R^2})$, we have $|T^{-1}(\rho(\varphi)(\cos\varphi,\sin\varphi))| = 1$, hence $|T^{-1}(\cos\varphi,\sin\varphi)| = 1/\rho(\varphi)$. For $x = r(\varphi)(\cos\varphi,\sin\varphi) \in S_X$,
  $$
    |T^{-1}x| = \frac{r(\varphi)}{\rho(\varphi)} = \exp (f(\varphi)-g(\varphi)),
  $$
  and the problem \eqref{eq:t-1-prob} is equivalent to
  $$
    \begin{cases}
      \max_\varphi \exp (f(\varphi)-g(\varphi)) \to \min\\
      g\in\mathcal{C},\quad g \le f.
    \end{cases}
  $$
  Since the exponential is strictly monotonically increasing, Proposition~\ref{prop:equiv} is equivalent to Proposition~\ref{prop:C-cal-equiv} below. There, for a $\pi$-periodic function $h$
  $$
    \|h\|_\infty := \max_\varphi |h(\varphi)|.
  $$
  \begin{pro}
    \label{prop:C-cal-equiv}
    The optimisation problem
    \begin{equation}
      \label{eq:C-cal-equiv}
      \begin{cases}
        \|f - g\|_\infty \to \min\\
        g\in\mathcal{C},\quad g \le f
      \end{cases}
    \end{equation}
    has a unique solution $\bar g$ characterised by the following property:

    There are points
    \begin{equation}
      \label{eq:alternance-pts}
      \varphi_1 < \psi_1 < \varphi_2 < \psi_2 < \varphi_1 + \pi
    \end{equation}
    such that for $i=1,2$
    \begin{equation}
      \label{eq:alternance-pts-values}
      f(\varphi_i) - \bar g(\varphi_i) = \|f-\bar g\|_\infty\text{ and }\bar g(\psi_i) = f(\psi_i).
    \end{equation}
  \end{pro}
  We will prove this proposition in the next section.

\section{Chebyshev Alternance of second order on the circle}
  \label{sec:Chebyshev}
  Let $C_\pi(\R)$ be the Banach space of all continuous $\pi$-periodic functions on $\R$ with the norm $\|\cdot\|_\infty$. 

  Let $\mathfrak R$ be the function that maps the triple $(a,b,c)\in\R^3$ to $\mathfrak {R}(a,b,c) \in C_\pi(\R)$ given by
  $$
    \mathfrak {R}(a,b,c) := a + b\cos 2\varphi + c\sin 2\varphi,
  $$  
  see \eqref{eq:rho-trig}. So, if $K$ is the cone given by \eqref{eq:ellipse-cond} then
  $$
    \mathcal{C} = \{-(1/2)\log\mathfrak{R}(a,b,c) : (a,b,c) \in K\}.
  $$

  Obviously, $\mathcal{C}$ is invariant under addition of a constant, that is,
  \begin{equation}
    \label{eq:inv-r}
    \forall r\in \R,\ \forall g \in \mathcal{C} \Rightarrow g+r \in \mathcal{C}.
  \end{equation}
  This fact allows considering the standard Chebyshev, that is, \textit{uniform}, approximation instead of the one-sided one \eqref{eq:C-cal-equiv}.

  \begin{pro}
    \label{prop:d=v/2}
    Let $\mathfrak{d}$ be the distance from $f$ to $\mathcal{C}$, that is, the value of
    \begin{equation}
      \label{eq:Ch}
      \begin{cases}
        \|f - g\|_\infty \to \min\\
        g\in\mathcal{C}.
      \end{cases}
    \end{equation}
    The function $\bar g \in \mathcal{C}$ is a solution to \eqref{eq:Ch} if and only if the function $(\bar g - \mathfrak{d})\in \mathcal{C}$ is a solution to \eqref{eq:C-cal-equiv}.
  \end{pro}
  \begin{proof}
    Let $\bar g$ be a solution to \eqref{eq:Ch} and let for short $g_1 := \bar g - \mathfrak{d}$. Since $\bar g-f \le \|\bar g-f\|_\infty = \mathfrak d$, we get $g_1 \le f$. Let $g_2\le f$ be a solution to \eqref{eq:C-cal-equiv} and let $c := \|f-g_2\|_\infty$.  
    
    Since  $g_1$ is feasible for \eqref{eq:C-cal-equiv}, we have $c \le \|f-g_1\|_\infty = \max (f-g_1) = \max (f-\bar g) + \mathfrak{d} \le 2\mathfrak d$. That is,
    \begin{equation}
      \label{eq:31}
      c \le \|f-g_1\|_\infty \le 2 \mathfrak{d}.
    \end{equation}

    On the other hand, if $g_3 := g_2 + c/2$ then $\mathfrak{d} \le \|f-g_3\|_\infty$. Let $\varphi\in\R$ be arbitrary. If $f(\varphi) \le g_3(\varphi)$ then $g_2(\varphi) \le f(\varphi) \le g_2(\varphi) + c/2$, so $|f(\varphi)-g_3(\varphi)| \le c/2$. If $f(\varphi) > g_3(\varphi)$ then $|f(\varphi)-g_3(\varphi)| = f(\varphi)-g_3(\varphi) = (f(\varphi)-g_2(\varphi)) - c/2 \le c/2$. So, 
    \begin{equation}
      \label{eq:32}
      \mathfrak d \le \|f-g_3\|_\infty \le c/2.
    \end{equation}

    From \eqref{eq:31} and \eqref{eq:32} it follows that $c=2\mathfrak d$, so \eqref{eq:31} gives that $g_1$ is a solution to \eqref{eq:C-cal-equiv}, and \eqref{eq:32} gives that $g_3$ is a solution to \eqref{eq:Ch}.
  \end{proof}

  We will prove in the present framework the well known geometrical theorem that at most one ellipse centered at the origin passes through three distinct points.

  \begin{lem}
    \label{lem:3-pts-determine}
    Let
    \begin{equation}
      \label{eq:in-pi-interval}
      \varphi_1 < \varphi_2 < \varphi_3 < \varphi_1 + \pi
    \end{equation}
    be fixed. Let $A:\R^3 \to \R^3$ be the linear operator
    $$
      A (a,b,c) := ( \mathfrak {R}(a,b,c)(\varphi_1),\mathfrak {R}(a,b,c)(\varphi_2),\mathfrak {R}(a,b,c)(\varphi_3)).
    $$
    Then $A$ is an isomorphism.
  \end{lem}
  \begin{proof}
    In co-ordinates
    $$
      A = \begin{pmatrix} 1 & \cos 2\varphi_1 & \sin 2\varphi_1 \\ 1 & \cos 2\varphi_2 & \sin 2\varphi_2 \\ 1 & \cos 2\varphi_3 & \sin 2\varphi_3 \end{pmatrix},
    $$
    so $\det A = 4\sin(\varphi_2 - \varphi_1)\sin(\varphi_3 - \varphi_1)\sin(\varphi_3 - \varphi_2) > 0$, since all three differences are in $(0,\pi)$.
  \end{proof}

  \begin{pro}
    \label{prop:elipses-3-pts-equial}
    Let $g_1,g_2\in\mathcal{C}$.

    If $g_1$ and $g_2$ coincide at three points as in \eqref{eq:in-pi-interval} then $g_1 = g_2$.

    If there are $\varphi_1, \varphi_2$ such that $0 < |\varphi_1-\varphi_2| < \pi$, and
    \begin{equation}
      \label{eq:g12-congr}
      g_1(\varphi_1) = g_2(\varphi_1),\  g_1'(\varphi_1) = g_2'(\varphi_1)\text{ and }g_1(\varphi_2) = g_2(\varphi_2),
    \end{equation}
    then $g_1 = g_2$.
  \end{pro}
  \begin{proof}
    Let $g_i = -(1/2)\log\mathfrak{R}(a_i,b_i,c_i)$ for $i=1,2$. If $g_1$ and $g_2$ coincide at three points as in \eqref{eq:in-pi-interval} then, since $\log$ is injective, $\mathfrak{R}(a_1,b_1,c_1)$ and $\mathfrak{R}(a_2,b_2,c_2)$ also coincide at those three points, so by Lemma~\ref{lem:3-pts-determine} we get $a_1=a_2$, $b_1=b_2$ and $c_1 = c_2$.

    Let now \eqref{eq:g12-congr} hold. Since
    $$
      g_i'(\varphi) = -\frac{1}{2}\left(\frac{d}{d\varphi}\mathfrak{R}(a_i,b_i,c_i)(\varphi)\right)/\mathfrak{R}(a_i,b_i,c_i)(\varphi),
    $$
    the three conditions translate to $\mathfrak{R}(a_1,b_1,c_1)$ and $\mathfrak{R}(a_2,b_2,c_2)$ agreeing at $\varphi_1$ in value and derivative, and at $\varphi_2$ in value. This implies 
    $$
      \begin{pmatrix}
        1 & \cos 2\varphi_1 & \sin 2\varphi_1 \\
        0 & -2\sin 2\varphi_1 & 2\cos 2\varphi_1 \\
        1 & \cos 2\varphi_2 & \sin 2\varphi_2
      \end{pmatrix}
      \begin{pmatrix} a_2-a_1 \\ b_2-b_1 \\ c_2-c_1 \end{pmatrix}
      = 0.
    $$
    The determinant of the above matrix is $4\sin^2(\varphi_2 - \varphi_1) > 0$, so $a_1=a_2$, $b_1=b_2$ and $c_1 = c_2$.
  \end{proof}
  
  We are ready to prove the uniqueness part of our main result.
  \begin{pro}
    \label{prop:unical}
    Let $\bar g\in \mathcal{C}$. Assume that there exist points $\varphi_i,\psi_i$, $i=1,2$, satisfying \eqref{eq:alternance-pts} and such that
    \begin{equation}
      \label{eq:alternance-pts-values-2}
      f(\varphi_i) - \bar g(\varphi_i) = \bar g(\psi_i) - f(\psi_i) =\|f-\bar g\|_\infty,\quad i=1,2.
    \end{equation}
    Then $\bar g$ is the unique solution to \eqref{eq:Ch}.
  \end{pro}
  \begin{proof}
    Let $\hat g \in \mathcal{C}$ be such that
    \begin{equation}
      \label{eq:51}
      \|f-\hat g\|_\infty \le \|f - \bar g\|_\infty.
    \end{equation}
    If we can prove that from this follows $\hat g = \bar g$ then obviously $\bar g$ is a solution to \eqref{eq:Ch} and, moreover, unique.
    
    Assume that
    \begin{equation}
      \label{eq:71}
      \bar g \neq \hat g.
    \end{equation}    
    Let $h := \bar g - \hat g$. We can write $h(\varphi_i) = (f-\hat g) (\varphi_i) - (f-\bar g) (\varphi_i)$. By \eqref{eq:alternance-pts-values-2} we have $(f-\bar g)(\varphi_i) = \|f-\bar g\|_\infty$, so $h(\varphi_i) = (f-\hat g)(\varphi_i) - \|f-\bar g\|_\infty \le \|f-\hat g\|_\infty - \|f-\bar g\|_\infty \le 0$ by \eqref{eq:51}. Thus $h(\varphi_i) \le 0$.
    Similarly, $h(\psi_i) = (f-\hat g)(\psi_i) + (\bar g -f)(\psi_i) = (f-\hat g)(\psi_i) + \|f-\bar g\|_\infty \ge \|f-\bar g\|_\infty  - \|f-\hat g\|_\infty \ge 0$. We have shown that
    \begin{equation}
      \label{eq:72}
      h(\varphi_i) \le 0 \le h(\psi_i),\quad i=1,2.
    \end{equation}
    Consequently, there exist $\xi_i\in[\varphi_i,\psi_i]$, $i=1,2$ (in particular, $\xi_1<\xi_2$) such that $h(\xi_i) = 0$ for $i=1,2$. If there were a third point in $[\varphi_1,\varphi_1+\pi)$ at which $h=0$, then Proposition~\ref{prop:elipses-3-pts-equial} would have implied $\bar g = \hat g$, contradicting \eqref{eq:71}, so
    \begin{equation}
      \label{eq:73}
      h(\varphi) \neq 0,\quad \forall \varphi \in [\varphi_1,\varphi_1+\pi) \setminus \{\xi_1,\xi_2\}.
    \end{equation}
    Also, if $h'(\xi_i) = 0$ for some $i=1,2$, then again Proposition~\ref{prop:elipses-3-pts-equial} implies $\bar g = \hat g$, so
    \begin{equation}
      \label{eq:74}
      h'(\xi_i) \neq 0,\quad i=1,2.
    \end{equation}
    From \eqref{eq:73} and $\xi_1\le\psi_1<\varphi_2\le\xi_2$ it follows that $h(\varphi) \neq 0$ for all $\varphi\in(\psi_1,\varphi_2)$. Since $h$ is continuous, it has a constant sign on $(\psi_1,\varphi_2)$, so there are two possibilities:

    \textsc{Case 1.} $h > 0$ on $(\psi_1,\varphi_2)$. Since $h(\varphi_2) \le 0$, see \eqref{eq:72}, we have $h(\varphi_2) = 0$, that is, $\xi_2=\varphi_2$. From \eqref{eq:74} it follows that $h'(\varphi_2) < 0$, so $h(\varphi_2+\delta) < 0$ for $\delta>0$ small enough. Since $h\neq0$ on $(\xi_2,\psi_2]$ by \eqref{eq:73}, we get $h(\psi_2)<0$, contradiction to \eqref{eq:72}.

    \textsc{Case 2.} $h < 0$ on $(\psi_1,\varphi_2)$. Since $h(\psi_1) \ge 0$, we have $\xi_1=\psi_1$, and $h'(\xi_1) < 0$, so $h$ is strictly positive to the left of $\xi_1$, implying $h(\varphi_1) > 0$, contradiction.
  \end{proof}

  \begin{pro}
    \label{prop:variation}
    Let $g\in\mathcal{C}$ be arbitrary. Let $\alpha < \beta < \alpha + \pi$. Let $\varepsilon > 0$.

    There exists $g_\varepsilon\in\mathcal{C}$ such that $\|g_\varepsilon-g\|_\infty < \varepsilon$ and
    $$
      g_\varepsilon(\varphi) < g(\varphi),\ \forall \varphi\in(\alpha,\beta),\quad g_\varepsilon(\varphi) > g(\varphi),\ \forall \varphi\in(\beta,\alpha+\pi).
    $$
  \end{pro}
  \begin{proof}
    Let $g = -(1/2)\log\mathfrak{R}(a,b,c)$ for some $(a,b,c)\in K$.
    Let $\gamma := (\alpha + \beta) / 2$. Let $A$ be the operator from Lemma~\ref{lem:3-pts-determine} for $\alpha,\gamma,\beta$ in place of the $\varphi$'s in \eqref{eq:in-pi-interval}.

    Let $(a_t,b_t,c_t) = A^{-1}(\mathfrak{R}(a,b,c)(\alpha),\exp (2t)\mathfrak{R}(a,b,c)(\gamma),\mathfrak{R}(a,b,c)(\beta))$. It is clear that for small enough $t > 0$ the point $(a_t,b_t,c_t) \in K$ and for $g_t := -(1/2)\log\mathfrak{R}(a_t,b_t,c_t)$
    $$
      g_t(\alpha) = g(\alpha),\ g_t(\beta) = g(\beta)\text{ and }g_t(\gamma) = g(\gamma) - t < g(\gamma).
    $$
    Since $(a_t,b_t,c_t)\to(a,b,c)$, as $t\to0$, and $\mathfrak{R}$ is continuous, we have $\|g-g_t\|_\infty < \varepsilon$ for $t$ small enough. Take $g_\varepsilon = g_t$ for some small enough $t$.

    If we assume that $g_\varepsilon(\xi) \ge g(\xi)$ for some $\xi\in(\alpha,\beta)$, then, since $g_\varepsilon(\gamma) < g(\gamma)$ there will be three different points in the interval $[\alpha,\beta]$ where $g_\varepsilon$ and $g$ coincide, which contradicts Proposition~\ref{prop:elipses-3-pts-equial}. So,
    $$
      g_\varepsilon(\varphi) < g(\varphi),\ \forall \varphi\in(\alpha,\beta).
    $$
    From this it follows in particular that $g_\varepsilon'(\beta) \ge g'(\beta)$. Since $g_\varepsilon'(\beta) = g'(\beta)$ would contradict Proposition~\ref{prop:elipses-3-pts-equial}, we have that $g_\varepsilon'(\beta) > g'(\beta)$. In particular $g_\varepsilon(\beta + \delta) > g(\beta+\delta)$ for all $\delta>0$ small enough. If there were $\xi\in(\beta,\alpha+\pi)$ such that $g_\varepsilon(\xi) \le g(\xi)$ this would contradict Proposition~\ref{prop:elipses-3-pts-equial}. So,
    $$
      g_\varepsilon(\varphi) > g(\varphi),\ \forall \varphi\in(\beta,\alpha + \pi).
    $$
  \end{proof}
  Of course, Proposition~\ref{prop:variation} can be proved by simply adding an appropriate trigonometric polynomial to $\mathfrak{R}$, but we chose to demonstrate a way that would work just as well for other families as long as they satisfy the interpolation Proposition~\ref{prop:elipses-3-pts-equial}. 
  
  \begin{pro}
    \label{prop:alternance}
    Let $\bar g$ be a solution to \eqref{eq:Ch}, that is, \begin{equation}
      \label{eq:alternance-*}
      \|f-\bar g\|_\infty = \mathfrak{d} = \min\{\|f-g\|_\infty:\ g\in \mathcal{C}\}.
    \end{equation}
    Then there exist points $\varphi_i,\psi_i$, $i=1,2$, satisfying \eqref{eq:alternance-pts} and \eqref{eq:alternance-pts-values-2}.
  \end{pro}
  \begin{proof}
    If $\mathfrak{d} = 0$ then $f=\bar g$ and any points satisfying \eqref{eq:alternance-pts} will do.

    Therefore, let $\mathfrak{d} > 0$ and
    $$
      \Phi := \{\varphi\in\R:\ f(\varphi) - \bar g(\varphi) = \mathfrak{d}\},\quad
      \Psi := \{\psi\in\R:\ \bar g(\psi) - f(\psi) = \mathfrak{d}\}.
    $$
    Because of the assumption $\mathfrak{d} > 0$ the sets $\Phi$ and $\Psi$ do not intersect.

    If $\Phi = \emptyset$ then $\max (f-\bar g) < \mathfrak{d}$. Consider $g_\varepsilon := \bar g - \varepsilon$. Because of \eqref{eq:inv-r} the functions $g_\varepsilon\in \mathcal{C}$. For $\varepsilon > 0$ small enough we still have $\max(f-g_\varepsilon) = \max(f-\bar g) + \varepsilon< \mathfrak{d}$, while $\min (f-g_\varepsilon) = \min(f-\bar g) + \varepsilon > - \mathfrak{d}$. So, $\|f-g_\varepsilon\|_\infty < \mathfrak{d}$, contradiction to \eqref{eq:alternance-*}. In a similar way, $\Psi\neq\emptyset$.

    Fix $\varphi_1\in\Phi$. Since $\Psi + \pi = \Psi$, because $f$ and $\bar g$ are $\pi$-periodic, and $\Psi\neq\emptyset$, we have $\Psi \cap [\varphi_1,\varphi_1+\pi]\neq\emptyset$. Let
    $$
      \psi_1 := \min \Psi \cap [\varphi_1,\varphi_1+\pi],\quad \psi_2 := \max \Psi \cap [\varphi_1,\varphi_1+\pi].
    $$
    Since $\varphi_1,\varphi_1+\pi\in\Phi$ and $\Phi\cap\Psi=\emptyset$, it follows that $\psi_i\in (\varphi_1,\varphi_1+\pi)$ for $i=1,2$.

    We claim that
    \begin{equation}
      \label{eq:alt-*}
      \Phi \cap [\psi_1,\psi_2] \neq \emptyset,
    \end{equation}
    (note Ader's Non-separation Principle).
    
    Assume \eqref{eq:alt-*} were false. Since $(f-\bar g)$ is a continuous function, which is $\mathfrak d$ at $\varphi_1$ and $-\mathfrak d$ at $\psi_1$, it vanishes somewhere in $(\varphi_1,\psi_1)$. Let $\alpha := \max \{\xi < \psi_1:\ f(\xi) = \bar g(\xi)\}$. In a similar way, let $\beta := \min\{\xi > \psi_2:\ f(\xi)=\bar g(\xi)\}$. We have then
    $$
      [\psi_1,\psi_2] \subset (\alpha,\beta),\quad [\alpha,\beta] \subset (\varphi_1,\varphi_1+\pi).
    $$
    Moreover, it is easy to see that
    $$
      \Phi\cap[\alpha,\beta] = \emptyset,
    $$
    because \eqref{eq:alt-*} was assumed false, and $f\le\bar g$ on $[\alpha,\psi_1]$ and $[\psi_2,\beta]$.

    By construction, $\Psi \cap ([\varphi_1,\alpha]\cup[\beta,\varphi_1+\pi]) = \emptyset$. Since $\Psi+\pi=\Psi$, the latter translates to
    $$
      \Psi \cap [\beta,\alpha+\pi] = \emptyset.
    $$
    So, there is $\varepsilon > 0$ such that
    $$
      \max_{[\alpha,\beta]} (f - \bar g) + \varepsilon < \mathfrak{d},\quad \max_{[\beta,\alpha+\pi]} (\bar g - f) + \varepsilon < \mathfrak{d}.
    $$
    From Proposition~\ref{prop:variation} there is $g_\varepsilon\in\mathcal{C}$ such that $\|g_\varepsilon - \bar g\|_\infty < \varepsilon$ and
    $$
      g_\varepsilon < \bar g\text{ on }(\alpha,\beta),\quad g_\varepsilon > \bar g\text{ on }(\beta,\alpha+\pi).
    $$
    We will show that
    \begin{equation}
      \label{eq:contra}
      \|f-g_\varepsilon\|_\infty = \max_{[\alpha,\alpha+\pi]}|f-g_\varepsilon| < \mathfrak d.
    \end{equation}
    Indeed, $\max_{[\alpha,\beta]}(f-g_\varepsilon) \le \max_{[\alpha,\beta]} (f-\bar g) + \|\bar g-g_\varepsilon\|_\infty < \mathfrak d$, and $\max_{[\alpha,\beta]}(g_\varepsilon - f) = g_\varepsilon(\xi) - f(\xi)$ for some $\xi\in[\alpha,\beta]$, but since $f=\bar g$ at $\alpha$ and $\beta$, and $\max_{[\alpha,\beta]}(g_\varepsilon - f) \ge g_\varepsilon(\psi_1) - f(\psi_1) = \mathfrak{d} + (g_\varepsilon(\psi_1)-\bar g(\psi_1)) \ge \mathfrak{d} - \|\bar g-g_\varepsilon\|_\infty > 0$; we have that $\xi\in(\alpha,\beta)$, thus $\max_{[\alpha,\beta]}(g_\varepsilon - f) = g_\varepsilon(\xi) - f(\xi) < \bar g(\xi) - f(\xi) \le \mathfrak d$.

    In a similar way, $\max_{[\beta,\alpha+\pi]}(f-g_\varepsilon) = f(\xi) - g_\varepsilon(\xi)$ for some $\xi\in (\beta,\alpha+\pi)$, so $< f(\xi) - \bar g(\xi) \le \mathfrak d$, and $\max_{[\beta,\alpha+\pi]}(g_\varepsilon-f) \le \max_{[\beta,\alpha+\pi]}(\bar g - f) + \|\bar g - g_\varepsilon\|_\infty < \mathfrak d$.

    Thus \eqref{eq:contra} is verified, but it contradicts \eqref{eq:alternance-*}. Therefore, there is $\varphi_2\in\Phi\cap(\psi_1,\psi_2)$ and $\varphi_i,\psi_i$ satisfy \eqref{eq:alternance-pts} and \eqref{eq:alternance-pts-values-2}.
  \end{proof}
  In view of the work so far the following completes our study.
  \begin{pro}
    \label{prop:final}
    The problem \eqref{eq:Ch} has a solution.
  \end{pro}
  \begin{proof}
    Here it is easier to work with the standard form of the ellipses in polar co-ordinates. Take $(a,b,c)\in K$, that is, $a>\sqrt{b^2+c^2}$, and set 
    $$
      b':= \sqrt{b^2+c^2},\quad b = b'\cos 2\theta,\ c = b'\sin 2\theta.
    $$
    Then
    $$
      \mathfrak{R}(a,b,c) = a + b'\cos 2(\varphi-\theta).
    $$
    $K$ translates to
    $$
      K' := \{(a,b',\theta):\ a>b'\ge0,\ \theta\in[0,\pi]\},
    $$
    and
    $$
      \mathcal{C} = \{-(1/2)\log(a + b'\cos 2(\varphi-\theta)) : (a,b',\theta) \in K'\}.
    $$
    Let $(a_n,b_n',\theta_n)\in K'$ be such that for
    $$
      g_n(\varphi) := -\frac{1}{2}\log(a_n + b_n'\cos 2(\varphi-\theta_n))
    $$
    we have
    $$
      \lim_{n\to\infty} \|f-g_n\|_\infty = \mathfrak{d}.
    $$
    By taking a subsequence we may assume that $\theta_n\to\theta$, as $n\to\infty$. Since
    $$
      g_n(\theta_n+\pi/4) = -\frac{1}{2}\log a_n,
    $$
    and $\|g_n\|_\infty < \|f\|_\infty + \mathfrak{d} + 1$ for large enough $n$'s, the sequence $\{\log a_n\}_{n=1}^\infty$ is bounded, so there is $r>0$ such that
    $$
      r \le a_n \le 1/r,\quad\forall n\in\mathbb{N}.
    $$
    Therefore, we may assume that $a_n\to a >0$, as $n\to\infty$. We also assume that $b_n'\to b'\le a$, as $n\to\infty$. If $b'=a$ then
    $$
      g_n(\theta_n+\pi/2) = -\frac{1}{2}\log(a_n - b_n') \to \infty,\text{ as }n\to\infty,
    $$
    which is a contradiction. So, $(a,b',\theta)\in K'$ and by continuity it follows that the corresponding $g\in\mathcal{C}$ solves \eqref{eq:Ch}.
  \end{proof}

\end{document}